\documentclass[12pt,reqno]{amsart}
\usepackage{soul}
\usepackage{multicol,amsmath,graphics, mathtools, cancel,soul,color,bbm,amsfonts,dsfont,tikz,mathrsfs,amssymb,bm}
\usepackage[left=1in, right=1in, top=1.1in,bottom=1.1in]{geometry}
\usetikzlibrary{matrix,patterns,positioning}
\numberwithin{equation}{section}

\newcommand{\E}{\mathbb{E}}

\newcommand{\N}{\mathbb{N}}
\newcommand{\Pb}{\mathbb{P}}

\newcommand{\R}{\mathbb{R}}

\newcommand{\Z}{\mathbb{Z}}

\newcommand{\vertiii}[1]{{\left\vert\kern-0.25ex\left\vert\kern-0.25ex\left\vert #1
    \right\vert\kern-0.25ex\right\vert\kern-0.25ex\right\vert}}

\def\R{\mathbb{R}}

\def\dh2l{\mathbf{d}_{\mathbb{H}_{2\ell}}}
\def\d2{\mathbf{d}_2}

\newcommand{\Indi}[1]{\mathbbm{1}_{{#1}}}

\newcommand{\Comb}[1]{\left(\begin{array}{c}#1\end{array}\right)}

\newtheorem{theorem}{Theorem}[section]
\newtheorem{corollary}[theorem]{Corollary}

\theoremstyle{remark}

\theoremstyle{definition}

\newcounter{dummy} \numberwithin{dummy}{section}
\newtheorem{Definition}[dummy]{Definition}

\newtheorem{Theorem}[dummy]{Theorem}
\newtheorem{Corollary}[dummy]{Corollary}

\newtheorem{Remark}[dummy]{Remark}

\def\1{\mathbbm{1}}

\makeatletter
\@namedef{subjclassname@1991}{2020 Mathematics Subject Classification}
\makeatother

\begin{document}

\title[Quantitative limit theorems via relative log-concavity]{Quantitative limit theorems via relative log-concavity}
\author{Arturo Jaramillo, James Melbourne}
\address{Arturo Jaramillo: Department of Probability and statistics, Centro de Investigaci\'on en matem\'aticas (CIMAT)}
\email{jagil@cimat.mx}
\address{James Melbourne: Department of Probability and statistics, Centro de Investigaci\'on en matem\'aticas (CIMAT)}
\email{james.melbourne@cimat.mx}

\date{\today}

 
\begin{abstract}
In this paper we develop tools for studying limit theorems by means of convexity. We establish bounds for the discrepancy in total variation between probability measures $\mu$ and $\nu$ such that $\nu$ is log-concave with respect to $\mu$. We discuss a variety of applications, which include geometric and binomial approximations to sums of random variables, and discrepancy between Gamma distributions.  As special cases we obtain a law of rare events for intrinsic volumes, quantitative bounds on proximity to geometric for infinitely divisible distributions, as well as binomial and Poisson approximation for matroids.

\end{abstract}

\maketitle

\section{Introduction}
In this manuscript we will study the proximity between probability measures with a perspective based in convexity. The theory we develop is designed to deviate from classical tools such as Fourier analysis, Stein's method or method of moments, in order to reach examples not accessible by the  mathematical technology currently available in the field of limit theorems. The families of distributions considered are those whose elements $\nu$ that are log-concave with respect to a reference measure $\mu$ either supported over $\Z$ or over $\R$, in the sense that the Radon-Nikodym derivative $\frac{d \nu}{d \mu}$ can be represented by $e^{-V}$ for a convex function $V$.\\

\noindent The topic of limit theorems has experienced significant growth in recent years,
much of which can be traced back to the celebrated paper ``A bound for the error in the normal approximation to the distribution of a sum of
dependent random variables" (see \cite{ChStein}), published by Charles Stein in 1972. This paper produced a seminal tool in the study of limit theorems: the so called “Stein’s method”, which nowadays refers to the collection of techniques that
allow bounding probability distances by means of differential operators. The reader is invited to consult the book \cite{ChenGoldShao} for a classic introduction to the topic and \cite{NoPe} for a perspective oriented to its application to Gaussian functionals. One should be aware that despite the 
outstanding reach of this technique, most of the hypothesis required for its implementation rely on a partial understanding of the dependence structure of the random variables under consideration, which is not obviously accessible in the setting we consider. \\

\noindent Let us motivate our setting. By choice of $\mu$, the log-concavity of the Radon-Nikodym derivative $\frac{d\nu}{d\mu}$ characterizes several important classes of random variables.  When $\mu$ is the counting measure on $\mathbb{Z}$, $\nu$ log-concave with respect to $\mu$ implies that $\nu$ itself is log-concave as a function on $\mathbb{Z}$. Note when  $\nu$ has non-negative support, this is equivalent to $\nu$ log-concave with respect to a geometric distribution $\mu[n] = (1-p)^n p$ for $p \in (0,1)$.  Such sequences seem to have been first systematically studied by Fekete \cite{fekete1912problem}, however it has been known since at least Newton, see \cite{stanley1989log, hardy1952inequalities, newton1732arithmetica} that the normalized  coefficients $p_n$ of a polynomial 
\begin{align} \label{eq: polynomial}
P(X) = \sum_{n=0}^m \binom{m}{n} p_n X^n
\end{align}
with real roots, satisfy $p_n^2 \geq p_{n+1}p_{n-1}$ and is thus log-concave in the sense above when $\{n: p_n >0 \}$ is a contiguous interval.  In this context the ``raw'' coefficients $a_n = \binom{m}{n} p_n$ in \eqref{eq: polynomial} possess  a much stronger form of concavity.  The sequence $a_n$ is log-concave with respect to a binomial distribution $\mu$, and referred to as ultra log-concave of order $m$ in the terminology of Pemantle \cite{pem2000} which we write as ULC$(m)$, see also \cite{borcea2009negative,liggett1997ultra}. The resolution of the strong Mason conjecture \cite{anari2018log, branden2020lorentzian}, yields that the cardinality of independent sets of size $n$ in a matroid with ground set of cardinality $m$ is ULC$(m)$. In the limit with $m \to \infty$ we recover the ``ultra log-concave'' distributions, namely, log-concave with respect to a Poisson distribution $\mu$, see \cite{johnson2007log, MR2598065, johnson2013log}.   By the Alexandrov-Fenchel inequality, the sequence of intrinsic volumes associated to a convex body form an ultra log-concave sequence \cite{schneider2014convex}, see also \cite{AMM22,lotz2020concentration} for recent work.   It remains an open question \cite{amelunxen2011geometric} what, if any, log-concavity properties the ``conic intrinsic volumes'' utilized in \cite{amelunxen2014living} possess.  Further classes defined through log-concavity with respect to a discrete  reference measure can be found in \cite{marsiglietti2020geometric}.   We mention in passing that log-concavity has long been associated with notions of negative dependence \cite{pem2000}, an analog and opposite theory to positive dependency \cite{esary1967association, fortuin1971correlation, ahlswede1978inequality}. Though we will not expand upon these connections here, in some sense the results here can be considered as limit theorems under the assumption of a kind of abstract negative dependency structure. 
In the continuous setting, when $\mu$ is the Lebesgue setting this corresponds (through the Pr\'ekopa-Liendler inequality) to the log-concave measures\footnote{A Radon measure $\mu$ is log-concave $\mu((1-t) A+t B) \geq \mu^{1-t}(A) \mu^{t}(B)$, as such log-concave and more generally $s$-concave measures \cite{borell1974convex, bobkov2016hyperbolic}, are defined on a topological vector space.}.  Let us also mention when $\mu$ is a standard Gaussian, that the measures log-concave with respect to $\mu$ are referred to as strongly log-concave (see \cite{saumard2014log}).  The strongly log-concave measures are exactly the measures that can be realized as the pushforward of said Gaussian by a $1$-Lipschitz measures \cite{caffarelli2000monotonicity}.\\



\noindent Taking this into consideration, we have developed a technique alternative to Stein's method, which could also allow for quantification of probability distances between log-concave probability distributions with respect to their reference measures, by leveraging convexity properties of the underlying densities in place of the dependence structure of the underlying measures, to answer the following question; is it possible to make a universal approximation of them with simpler distributions? in other words: is it possible to establish a limit theorem valid for every choice of log-concave distribution?\\

\noindent Let us outline the paper. After describing our main tool for bounding probability distances (Theorem \ref{Theoremmain:1}), we will present a series of applications both for classical examples in probability theory (such as sums of independent random variables) and for novel problems, arising in the context of random matroids and intrinsic volumes. For the classical examples, we show perspectives for tackling the approximation problems, which are alternative to those based on Stein's method, and yield new bounds for the associated error measured in the total variation metric. In particular, we would like to bring the readers attention to Corollaries \ref{cor:Geometricpoissonlthm} and \ref{cor:geomapproxtoID}, where we present assessments of the proximity of compound Poisson and compound geometric distribution towards a geometric random variable. In Corollary \ref{cor:poissonvspoissonbinom} we present a bound for the approximation of Poisson binomial random variables with a binomial distribution, which up to our knowledge, has not been found before. We prove as well bounds for geometric approximations in the context of sums of random variables (see Corollary \ref{cor:Geometricapprox}). Regarding the problems related to random matroids, we show a general bound for assessing the distance of these objects towards Binomial and Poisson distributions and specialize these results to the case of uniform and partitioned matroids. For the case of intrinsic volumes, we present a general bound for Poissonian approximations (similar in spirit to the matroid case) and then we use it for proving a law of rare events for the intrinsic volume random variables associated to products of convex sets (See Corollary \ref{cor:rareeventsintrinsvol}).
\noindent Section \ref{Proofmain} is dedicated to the proof of Theorem \ref{Theoremmain:1} while in Section \ref{sec:continuouspart} we present the continuous version of the results and some applications related to the discrepancy between Gamma distributions.

\section{The discrete regime}\label{eq:discreteregime}
%
%
In this subsection we present upper bounds on distance between two distributions when one is log-concave with respect to the other. The proof will be presented in Section \ref{Proofmain}.

Consider two (not necessarily finite) real measures $\mu$ and $\nu$, supported over $\Z$. We say that $\nu$ is log-concave with respect to $\mu$ if $\nu$ is absolutely continuous with respect to $\mu$ and the Radon-Nikodym derivative $\frac{d\nu}{d\mu}$ is such that the mapping $k\mapsto\Phi_{\mu,\nu}(k)$, with
\begin{align*}
\Phi_{\mu,\nu}(k)
  &:=\left\{\begin{array}{ccc}
\log( \frac{d\nu}{d\mu}[k])  & \text{ if } & \nu[k]\neq0\\
-\infty  & \text{ if } & \nu[k]=0
\end{array}\right.
\end{align*}
is concave in the sense that for all $i,j\in\Z$, and $\lambda\in[0,1]$ satisfying $(1-\lambda)i+\lambda j\in\Z$, 
\begin{align}\label{eq:convexityofPhi}
\Phi_{\mu,\nu}((1-\lambda)i+\lambda j)
  &\geq (1-\lambda)\Phi_{\mu,\nu}(i)+\lambda\Phi_{\mu,\nu}(j).
\end{align}

\begin{Theorem}\label{Theoremmain:1}
Suppose that $\nu$ and $\mu$ are probability measures and $\nu$ is log-concave with respect to $\mu$ and the support of $\mu$ is connected. Define 
$q_{k}=\nu[k]$ and $p_{k}:=\mu[p]$. Then, for all $\ell\in\Z$ satisfying $q_{\ell}q_{\ell+1}>0$, we  have the bounds
\begin{align*}
d_{TV}(\mu,\nu)
  &\leq \int_{\Z}\bigg{(}1-
	\frac{p_{\ell}}{q_{\ell}}\left(\frac{p_{\ell+1}q_{\ell}}{p_{\ell}q_{\ell+1}}\right)^{|y-\ell|}
	\bigg{)}_{+}\nu(dy).
\end{align*}
and 
\begin{align*}
d_{TV}(\mu,\nu)
  &\leq \int_{\Z}\bigg{(}
	\frac{q_{\ell}}{p_{\ell}}\left(\frac{p_{\ell}q_{\ell+1}}{p_{\ell+1}q_{\ell}}\right)^{|y-\ell|}-1
	\bigg{)}_{+}\mu(dy).
\end{align*}
In particular, if $\ell$ satisfies the ratio identity $\frac{p_{\ell}}{p_{\ell+1}}=\frac{q_{\ell}}{q_{\ell+1}}$, then we have that $\frac{p_{\ell}}{q_{\ell}}\geq 1$ and the above bound simplifies to 
\begin{align*}
d_{TV}(\mu,\nu)
  &\leq (\frac{q_{\ell}}{p_{\ell}}-1)\wedge(1-\frac{p_{\ell}}{q_{\ell}}).
\end{align*}

%
%
%
%
%
%

\end{Theorem}
As mentioned in the introduction, tha above result applies to the following settings 
\begin{enumerate}
\item[-] Binomial and Poisson approximations to Poisson binomial distributions.
\item[-] Binomial approximations to random matroids.
\item[-] Poisson approximations to random intrinsic volumes.
\item[-] Geometric approximation to compund Poisson and compound geometric distributions.
\end{enumerate}
The following corollary that relates the law of a Poisson binomial distribution and a Binomial distribution.

\begin{Corollary}[Binomial approximation to Poisson binomial]\label{cor:poissonvspoissonbinom}
Let $\{\xi_i\}_{i\geq1 }$ be a collection of Bernoulli random variables with $\alpha_{i}:=\Pb[\xi_{i}=0]$ and $p_i:=\Pb[\xi_{i}=1]$. For a given $n\geq 1$,  define 
$S_{n}:=\sum_{i=1}^{n}\xi_i$, namely, $S_{n}$ is distributed according to the so called Poisson-Binomial distribution of parameters $p_1,\dots, p_n$. Consider the arithmetic mean $\mathfrak{m}_n$ of the sequence $\alpha_{1}^{-1},\dots, \alpha_n^{-1}$ and let $\mathfrak{g}_n$ be associated geometric mean, defined by 
\begin{align*}
\mathfrak{g}_{n}
  &:=(\prod_{i=1}^{n}\alpha_{i}^{-1})^{1/n}.
\end{align*}
Then, if $B_{n}$ is a Binomial distribution with parameters $n$ and $1-1/\mu_{n}$, we have that 
\begin{align}\label{eq:dTVBinomPoissonbinom}
d_{TV}(S_{n},B_{n})
  &\leq ((\mathfrak{m}_n/\mathfrak{g}_{n})^{n}-1)\wedge(1-(\mathfrak{g}_{n}/\mathfrak{m}_n)^{n}).
\end{align}
Moreover, 
\begin{align}\label{eq:dTVBinomPoissonbinom2}
d_{TV}(S_{n},B_{n})
  &\leq  \exp\{\sum_{i=1}^{n}(p_i/\alpha_i-\mathfrak{r}_n)^2
  +\frac{1}{3n^2}\sum_{i=1}^n(p_i/\alpha_i)^3
  \}-1,
\end{align}
where $\mathfrak{r}_n$ denotes the arithmetic mean of the terms $p_i/\alpha_i$,
\begin{align}\label{eq:Amathfrakdef}
\mathfrak{r}_n
  &:=\frac{1}{n}\sum_{i=1}^n\frac{p_i}{\alpha_i}
\end{align}
\end{Corollary}
The following observations are relevant for comparing the above result with the previous literature
\begin{Remark}
The term \eqref{eq:Amathfrakdef} is close in spirit to the bounds from \cite[Theorem 1]{MR1093412}, where the Poisson Binomial distribution  $S_n$ is approximated by a binomial law with a slightly different (and perhaps, more popular) success probability 
$\mathfrak{p}_n=\sum_{i=1}^np_i$. Here the upper bound for the associated distance in total variation is
\begin{align*}
(1-\mathfrak{p}_n^n)(1-(1-\mathfrak{p}_n)^n)((n+1)\mathfrak{p}_n(1-\mathfrak{p}_n))^{-1}\sum_{i=1}^{n}(p_i-\mathfrak{p}_n)^2.
\end{align*}
We would like to highlight the analogy of the terms $\sum_{i=1}^{n}(p_i-\mathfrak{p}_n)^2$ and $\mathfrak{r}_n$, as both measure the deviation of a sequence of parameters ($p_i$ and $p_i/\alpha_i$, respectively) around their mean.
\end{Remark}

\begin{Remark}
In view of the remark above, one observes that the classical methodology for choosing the appropriate parameter $p$ of the Binomial approximation (which consists on "matching in the  expectations"), is not the only natural one, as the procedure applied in corollary \ref{cor:poissonvspoissonbinom} (which is based on the idea of matching quotients of probabilities), also leads to good approximation results. One can thus pose the following question, which up to our knowledge has not yet been answered\\

\noindent \textit{Open problem}\\
What are the optimal parameters $m,p$ for approximating a Poisson binomial distribution with a Binomial distributions, in the sense: if $B_{m,p}$ is distributed according to a Binomial distribution of parameters $m,p$, then
\begin{align*}
d_{TV}(B_{m,p},S_{n})
  &\leq d_{TV}(B_{m^{\prime},p^{\prime}},S_{n})
\end{align*}
for any choice of $(m^{\prime},p^{\prime})\in\N\times[0,1]$?
\end{Remark}

\begin{proof}[Proof of Corollary \ref{cor:poissonvspoissonbinom}]
By Liggett \cite{liggett1997ultra}, $S_{n}$ is log-concave with respect to $B_{n}$. Thus, by Theorem \ref{Theoremmain:1}, it suffices to verify $\Pb[S_{n}=1]/\Pb[S_{n}=0]=\Pb[B_{n}=1]/\Pb[B_{n}=0]$ and estimate $(\frac{\Pb[S_n=0]}{\Pb[B_n=0]}-1)\wedge(1-\frac{\Pb[B_n=0]}{\Pb[S_n=0]})$.\\

\noindent Next we prove that $\Pb[S_{n}=1]/\Pb[S_{n}=0]=\Pb[B_{n}=1]/\Pb[B_{n}=0]$. To show this notice, that 
\begin{align}\label{eq:Snprobtozeroone}
\Pb[S_{n}=1]
  &=\sum_{i=1}^np_i\prod_{j\neq i} \alpha_{j}
	=\sum_{i=1}^n\frac{p_i}{\alpha_i}\prod_{j=1}^n \alpha_{j}
	=\Pb[S_{n}=0]\sum_{i=1}^n\frac{p_i}{\alpha_i}.
\end{align}
so we can write
\begin{align}\label{eq:Snprobtozerooneprime}
\Pb[S_{n}=1]
  &=\Pb[S_{n}=0]n(\mathfrak{m}_n-1).
\end{align}
Specializing this identity in the case where $\alpha_{i}=1/\mathfrak{m}_n$ for all $i=1,\dots, n$ yields
\begin{align*}
\Pb[B_{n}=1]
  &=\Pb[B_{n}=0]n(\mathfrak{m}_n-1),
\end{align*}
which yields $\Pb[S_{n}=1]/\Pb[S_{n}=0]$. By Theorem \ref{Theoremmain:1}, we thus conclude that 
\begin{align*}
d_{TV}(S_{n},B_{n})
  &\leq (\Pb[S_{n}=0]/\Pb[B_{n}=0]-1)\wedge(1-\Pb[B_{n}=0]/\Pb[S_{n}=0]),
\end{align*}
which reads 
\begin{align*}
d_{TV}(S_{n},B_{n})
  &\leq (\mathfrak{m}_n^{n} \prod_{i}a_i-1)\wedge(1-\mathfrak{m}_n^{-n} \prod_{i}a_i^{-1})
\end{align*}
Identity \eqref{eq:dTVBinomPoissonbinom} follows from here.\\

\noindent To prove \eqref{eq:dTVBinomPoissonbinom2}, we observe that 

\begin{align}\label{eq:lofquofofmeqands}
\log((\mathfrak{m}_n/\mathfrak{g}_n)^n)
  &=n(\log(\frac{1}{n}\sum_{i=1}^{n}\frac{1}{\alpha_i})-\sum_{i=1}^{n}\frac{1}{n}\log(\frac{1}{\alpha_i}))\\
  &=n(\log(1+\frac{1}{n}\sum_{i=1}^{n}\frac{p_i}{\alpha_i})-\sum_{i=1}^{n}\frac{1}{n}\log(1+\frac{p_i}{\alpha_i}))\nonumber.
\end{align}
By Taylor approximation, we deduce that for every $x>-1$,
\begin{align*}
\log(1+x)
  &=x-\frac{x^2}{2}+2\int_{[0,1]^3}\frac{x^3}{(1+\lambda_1\lambda_2\lambda_3x)^3}\lambda_2\lambda_1^2d\lambda_3d\lambda_2d\lambda_1.
\end{align*}
In particular, 
\begin{align}\label{ineq:logTaylor}
x-\frac{x^2}{2}
  \leq \log(1+x)
  \leq x-\frac{x^2}{2}+\frac{x^3}{3}.
\end{align}
Applying these inequalities in \eqref{eq:lofquofofmeqands}, we deduce that 
\begin{align*}
\frac{1}{n}\log((\mathfrak{m}_n/\mathfrak{g}_n)^n)
  &\leq 
  \frac{1}{2}\sum_{i=1}^{n}\frac{1}{n}\left(\frac{p_i}{\alpha_i}\right)^2-\frac{1}{2}\left(\frac{1}{n}\sum_{i=1}^{n}\frac{p_i}{\alpha_i}\right)^2+\frac{1}{3}\left(\frac{1}{n}\sum_{i=1}^{n}\frac{p_i}{\alpha_i}\right)^3.
\end{align*}
The sum of the two terms in the right equals to the variance of the mapping $i\mapsto p_i/\alpha_i$ under the normalized counting measure over $\{1,\dots, n\}$. Consequently, 
\begin{align*}
\log((\mu_n/\mathfrak{g}_n)^n)
  &\leq 
  \frac{1}{2}\sum_{i=1}^{n}\left(\frac{p_i}{\alpha_i}-\mathfrak{r}_i^n\right)^2+\frac{1}{3n^2}\left(\sum_{i=1}^{n}\frac{p_i}{\alpha_i}\right)^3.
\end{align*}
The result easily follows from here.
\end{proof}

Next we present results related to the Poisson approximation to the Poisson binomial distributions. This topic has been extensively studied with a variety of techniques, and its development started with the results by Prohorov in \cite{MR0116370}, later refined by Kerstan, Vervaat. Hodges and Le Cam (see \cite{MR324784}, \cite{MR242235}, \cite{MR117812} and \cite{MR142174}). These results where subsequently addressed with a Stein's method perspective in the groundbreaking paper \cite{MR428387} by Chen. the next result describes a bound with a shape quite similar to the optimal one, presented for instance in the papers \cite{MR142174} and \cite{MR428387}, but proved with a perspective completely different. One should observe as well that the parameter that we use for the approximation is different from the previous literature, so a difference in the shape of the bound should be expected.

\begin{Corollary}[Poisson approximation to Poisson Binomial]
Let $S_{n}$ and $\xi_{i}$ be as in Corollary \ref{cor:poissonvspoissonbinom}. Let $N_{n}$ be a Poisson distribution with parameter $\lambda_n:=n(\mathfrak{m}_n-1)$, then 
\begin{align*}
 d_{TV}(S_{n},M_{n})
   &\leq  \exp\{\sum_{i=1}^n(p_i/\alpha_i)^2\}-1.
\end{align*}
\end{Corollary}
\begin{Remark}
By choosing $p_{i}=r/n$, for some $r>0$, we obtain the  quantitative bound of the law of rare events, 
\begin{align*}
 d_{TV}(S_{n},M_{n})
   &\leq  \exp\left\{n\left(\frac{r}{n-r}\right)^2\right\}-1
   =O(1/n),
\end{align*}
which has been studied quite successfully with Stein's method techniques. 
\end{Remark}

\begin{proof}
Note that $S_n$ is ultra log-concave, as $\xi_i$ are ultra log-concave, and ULC variables are stable under independent summation \cite{walkup1976polya}. Moreover, by the choice of $\lambda_{n}$ and \eqref{eq:Snprobtozerooneprime}, we have that 
\begin{align*}
\frac{\Pb[S_{n}=1]}{\Pb[S_{n}=0]}
  &=\frac{\Pb[N_{n}=1]}{\Pb[N_{n}=0]}, 
\end{align*}
so by Theorem \ref{Theoremmain:1}, 
\begin{align*}
 d_{TV}(S_{n},M_{n})
   &\leq \frac{\Pb[S_{n}=0]}{\Pb[M_{n}=0]}-1
   = e^{\lambda}\mathfrak{m}_n^{-1}-1.
\end{align*}
Proceeding as in \eqref{eq:lofquofofmeqands}, we can show that
\begin{align*}
\log(e^{\lambda}\mathfrak{m}_n^{-1})
  &=n(\frac{1}{n}\sum_{i=1}^{n}\frac{p_i}{\alpha_i}-\sum_{i=1}^{n}\frac{1}{n}\log(1+\frac{p_i}{\alpha_i}))\nonumber.
\end{align*}
Thus, by \eqref{ineq:logTaylor}, 
\begin{align*}
\log(e^{\lambda}\mathfrak{m}_n^{-1})
  &\leq \sum_{i=1}^{n}\left(\frac{p_i}{\alpha_i}\right)^2.
\end{align*}
The result easily follows from here.
\end{proof}

\subsection{Applications to matroids}\label{sec:randommatroids}
In this section we present approximation results for  matroids. Although the methodology utilized is identical to the previous subsection, the definitions of the objects of interest require a non-negligible amount of notation. A finite matroid $\mathcal{M}$ is a pair $(X,\mathcal{I})$ of a ground set $X$ with $|X|=n$ and a nonempty collection of idependent sets $\mathcal{I}\subset 2^{X}$ that satisfies the following 

\begin{enumerate}
\item[(i)] (Hereditary property) If $S\subset T$ and $T\in\mathcal{I}$, then $S\in\mathcal{I}$, 
\item[(ii)] (Exchange property)	If $S,T\in\mathcal{I}$ and $|S|<|T|$, then there exists an element $x\in T\backslash S$ such that $S\cup\{x\}\in \mathcal{I}$.
\end{enumerate}  
The rank $rk(\mathcal{M})$ of a matroid is the maximal size of the independent sets of $S$, namely, $rk(\mathcal{M}):=\max_{S\in\mathcal{I}}|S|$. Finally, we define $\mathcal{I}_{k}:=\{S\in\mathcal{I}\ ;\ |S|=k\}$ and let $I(k):=|\mathcal{I}_{k}|$, that is to say, $I(k)$ is the number of independent sets of size $k$. Countings related to $I(k)$ can be calculated by the introduction of the measure $\nu_{\mathcal{M}}$, supported in $\{1,\dots, rk(\mathcal{M})\}$ and with 

\begin{align}\label{eq:matroidrv}
\nu_{\mathcal{M}}[\{k\}]
   :=(\sum_{j=1}^nI(j))^{-1}I(k).
\end{align}
The next result describes the proximity of $\nu_{\mathcal{M}}$ and a binomial distribution 
\begin{Corollary}\label{eq:matroidbinom}
Let $\gamma$ be the probability law of a random variable $B$ with Binomial distribution of sample size $n=|X|$ and success probability 
$p=(1+\frac{n-m}{m+1}\frac{I(m)}{I(m+1)})^{-1}$.
Then, for every $m\leq n-1$, We have that 
\begin{align*}
d_{TV}(\nu_{\mathcal{M}},\gamma)
  &\leq \frac{I(m)}{\Comb{n\\m}p^{m}(1-p)^{n-m}\sum_{j=1}^nI(j)}-1
\end{align*}
and 
\begin{align*}
d_{TV}(\nu_{\mathcal{M}},\gamma)
  &\leq 1-\frac{1}{I(m)}\Comb{n\\m}p^{m}(1-p)^{n-m}\sum_{j=1}^nI(j)
\end{align*}
\end{Corollary}

\begin{proof}
By the resolution of the strong Mason conjecture  \cite{AnLiuGhVin, branden2020lorentzian},  the sequence $I(k)$ satisfies 
\begin{align}\label{eq:Ulogconcvmatroids}
I(k)^2
  &\geq \left(1+\frac{1}{k}\right)\left(1+\frac{1}{n-k} \right)I(k-1)I(k+1).
\end{align}
 Note by direct computation, inequality \eqref{eq:Ulogconcvmatroids} is equivalent to the log-concavity of the sequence $\{x_k \}_{k=0}^n$, defined by $x_k \coloneqq \frac{I(k)}{b_k}$ where $b_k = \binom{n}{k} (1-p)^{n-k}p^k$ for $p \in (0,1)$.  That is $I(k)$ is log-concave with respect to $B \sim$ binomial$(p,n)$.  We can easily check that the equality is attained if for any sequence of terms $I(k)$ are replaced by constant multiples of the atoms of a binomial distribution.
We can easily check that  
\begin{align*}
\Pb[B_{n}=m+1]/\Pb[B_{n}=m]
  &=\frac{n-m}{m+1}(p/(1-p)).
\end{align*}
By definition of $p$, the right hand side coincides with  $\nu[2]/\nu[1]=I[2]/I[1]$ so that $\nu[2]/\nu[1]=\gamma[2]/\gamma[1]$. Corollary \ref{eq:matroidbinom} then follows from Theorem \ref{Theoremmain:1}.
\end{proof}

By similar arguments, we can prove the following Poisson approximation
\begin{Corollary}\label{Corollary:Poissonmatroidsapprox}
Let $\gamma$ be the probability law of a random variable $M$ with Poisson distribution of parameter 
$\lambda=(m+1)\frac{I(m+1)}{I(m)}$.
Then we have that 
\begin{align*}
d_{TV}(\nu_{\mathcal{M}},\gamma)
  &\leq \frac{m!e^{\lambda}I(m)}{\lambda^m\sum_{j=1}^nI(j)}-1
\end{align*}
\end{Corollary}

Some basic examples of matroids defined over a grounds $X$ of cardinality $n$, for which the above approximations hold are the following\\

\noindent\textit{The uniform matroid}\\
For a given $1\leq r\leq n$, we define the uniform matroid $\mathcal{M}_U$ as the one whose independent sets are exactly those subsets of $X$  of at most $r$ elements.\\

\noindent\textit{The partition matroid}\\
Consider a partition of $X$ consisting of subsets $C_1,\dots, C_\ell$ and integers $d_1,\dots, d_{\ell}$, with $0\leq d_{i}\leq |C_{\ell}|$. The sets $C_{i}$ and the integers $d_i$ are called 'categories' and 'capacities', respectively. A set $I\subset X$ to belong to the set of independent sets $\mathcal{I}_P$ if
\begin{align*}
|I\cap C_{i}|
  &\leq d_{i}.
\end{align*}
We will denote the  associated matroid by $\mathcal{M}_P$. Observe that uniform matroids are a particular instance of partition matroids\\

\noindent The following result easily follows from the above discussion 
\begin{corollary}
Let $\tau$ be the probability measure induced by $\mathcal{M}_P$, according to \eqref{eq:matroidrv}. Define $D_{U,n}$ and $D_{M,n}$ as the cardinality of $2^{n}\backslash\mathcal{I}_{\mathcal{M_{U}}}$ and $2^{n}\backslash\mathcal{I}_{\mathcal{M_{P}}}$ respectively. Assume that the capacities of $\mathcal{M}_P$ and the parameter $k$ in the definition of uniform matroid $\mathcal{M_{U}}$ are such that $\min_{i=1,\dots, \ell}|C_{i}|\wedge d_{i}\geq 2$. Then,
\begin{align}\label{eq:dTVDPN}
d_{TV}(\tau,\rho)
  &\leq (1-2^{-n}D_{P,n})^{-1}-1,
\end{align}
where $\rho$ is a Binomial distribution of sample size $n$ and success probability $1/2$. In the particular case where $\mathcal{M}_{P}$ is a uniform matroid with $k\geq n-\varepsilon\frac{n}{\log_2( n)}+1$, for some $\varepsilon\in(0,1)$ satisfying $(1-\varepsilon)n\geq 1$, then the above inequality reduces to
\begin{align}\label{eq:dTVDPN2}
d_{TV}(\tau,\rho)
  &\leq 2^{2-(1-\varepsilon)n}.
\end{align}

\end{corollary}

\begin{proof}
Denote by $I_{P}(k)$ and $I_{U}(k)$ the set of independent sets of size $k$ according to the matroids $\mathcal{M}_P$ and $\mathcal{M}_U$ respectively. Observe that 
\begin{align}\label{ineq:sumpartmatroid}
\sum_{k=1}^{n}I_{P}[k]
  &=2^{n}-D_{P,n}.
\end{align}
Since all the capacities $d_{1},\dots, d_{\ell}$ are all greater than or equal to two, every subset of $X$ of cardinality one or two is independent, and thus,
\begin{align*}
I_{P}(k)
  &=\Comb{n\\k}
\end{align*}
for $k=1,2$. From here it easily follows that if $m=1$, the parameter $p$ defined according to Corollary $\ref{eq:matroidbinom}$ is equal to $1/2$, so that \begin{align*}
d_{TV}(\nu_{\mathcal{M}_P},\gamma)
  &\leq \frac{I_{P}(1)}{\Comb{n\\1}2^{-n}\sum_{j=1}^nI(j)}-1
	=2^{n}(2^{n}-D_{P,n})^{-1}.
\end{align*}
Relation \eqref{eq:dTVDPN} easily follows from here. To deduce \eqref{eq:dTVDPN2}, we find a suitable upped bound for $D_{U,n}$. To this end, we observe that after suitable algebraic manipulations,
\begin{align*}
D_{P,n}
  &=\sum_{j=k+1}^{n}\frac{n!}{k!(n-j)!}
	\leq n^{n-k+1}.
\end{align*}
Thus, using the condition $k\geq n-\varepsilon\frac{n}{\log_2(n)}+1$, we obtain
\begin{align*}
\log_2(2^{-n}D_{P,n})
  &\leq  -(1-\varepsilon)n,
\end{align*}
so that 
\begin{align}
d_{TV}(\tau,\rho)
  &\leq (1-2^{(1-\varepsilon)n})^{-1}-1.
\end{align}
Relation \eqref{eq:dTVDPN2} follows from here, due to the condition $(1-\varepsilon)n\geq 1$.
\end{proof}

\subsection{Applications to intrinsic volumes}\label{sec:intrinsicvol}
In this section we present approximation results for intrinsic volumes, which are measures of the content of a convex body through applications of Theorem \ref{Theoremmain:1}. Throughout this section, $n$ will denote a natural number that will indicate the ambient dimension into which a convex set of interest will be embedded. $K\subset \R^{n}$ will denote a nonempty convex set in $\R^{n}$. The dimension of $K$, denoted by $dim(K)$ is the dimension of the affine hull of $K$, which can take values in $0,1,\dots, n$. When $K$ has dimension $l$, we define the $l$-dimensional volume $Vol_{l}(K)$ to be the Lebesgue measure of $K$, computed relative to its affine hull. If $K$ is zero dimensional, then we define $Vol_0(K)=1$. For sets $A\subset\R^{a}$ and $B\subset \R^{b}$, then 
\begin{align*}
A\times B
  &=\{(a,b)\in\R^{a+b}\ ;\ a\in A\text{ and } b\in B\}.
\end{align*} 
The unit-volume cube $[0,1]^n$ will be denoted by $Q_{n}$ and the unit ball in $\R^{n}$ will be denoted by $\R^{n}$. The volume $\kappa_{n}$ is given by the formula 
\begin{align*}
\kappa_{n}
  &=Vol_{n}(B_{n})=\frac{\pi^{n}}{\Gamma(1+n/2)},
\end{align*}
where $\Gamma$ denotes the Gamma function. 

\begin{Definition}[Intrinsic volumes]
For each $j=0,...,n$, let $\bm{P}_{j}\in\R^{n\times n}$ be the orthogonal projection over a given $j$-dimensional subspace of $\R^{n}$. Let $\bm{Q}\in\R^{n\times n}$ be a Haar distributed element from the set of orthogonal matrices with determinant equal to one. The intrinsic volume of order $j$ of $K$ is defined as 
\begin{align*}
V_{j}
  &:=\Comb{n\\j}\frac{\kappa_{n}}{\kappa_{j}\kappa_{n-j}}\E[Vol_{j}(\bm{P}_{j}\bm{Q}(K))].
\end{align*}
The intrinsic volumes of an empty set will be defined to be zero. The intrinsic volumes satisfiy the dilation property 
$V_{j}(\varepsilon K)=\varepsilon^jV(K)$ as well as the direct product property, which states that if $A\subset\R^{a}$ and $B\subset\R^{b}$, for some convex sets $A,B$ and $a,b\geq 1$, then $W(A\times B)=W(A)W(B)$.
\end{Definition}
Several properties of intrinsic volumes of $K$ can be studied by means of probabilistic arguments, via the so called 'intrinsic volume random variable' associated to $K$. 
\begin{Definition}
We say that a random variable $Z_{K}$ has the intrinsic volume distribution according to $K$, if for $j=0,\dots, n$,
\begin{align*}
\Pb[Z_{K}=j]
  &=\frac{V_{j}(K)}{W(K)},
\end{align*}
where $W(K) :=\sum_{i=0}^{n}V_{j}$ is the total intrinsic volume.
\end{Definition}
 
\noindent Some instances in which the above definitions can be computed explicitly are the following:\\

\noindent \textit{The Euclidean ball}\\
When $K$ is the unit ball of $\R^{n}$, then 
\begin{align*}
V_{j}(K)
  &=\Comb{n\\ k}\frac{\kappa_{n}}{\kappa_{n-j}}.
\end{align*}

\noindent \textit{The cube}\\
When $K$ is scaled cube 
\begin{align}\label{eq:Kcubedeffor}
K
  &:=\{(x_1,\dots, x_{n})\in\R^{n}\ ;\ |x_{i}|\in [0,s]\},
\end{align}
then 
\begin{align*}
V_{j}(K)
  &=s^{j}\Comb{n\\ k},
\end{align*}
and the intrinsic volume random variable $Z_{K}$ has Binomial distribution of size $n$ and success probability $s/(1+s)$.\\

\noindent \textit{Rectangular paraleletope}\\
Let $s_1,\dots, s_{n}\geq0$ be fixed. When $K$ is given by
\begin{align*}
K
  &:=\{x\in\R^{n}\ ;\ x_i\in[0,s_{i}]\},
\end{align*}
then $V_{j}(K)$ is the $j$-th symmetric function  and the total intrinsic volume satisfies
\begin{align*}
W(K)
  &=\prod_{i=1}^{n}(1+s_{i}).
\end{align*}

The next result describes the proximity of an intrinsic volume random variable towards a Poisson law. 
\begin{Corollary}\label{Cor:PoissonIVgenericresult}
Let $\gamma_K$ be the probability law of an intrinsic volume random variable $Z_{K}$, where $K\subset \R^{n}$ is a compact set. Let $\nu_{K}$ be the Poisson probability measure of parameter $\lambda=(m+1)\frac{V_{m+1}(K)}{V_{m}(K)}$, for $0\leq m\leq n-1$. Then we have that 
\begin{align}\label{eqdTVforivgen}
d_{TV}(\nu_{K},\gamma_K)
  &\leq \frac{m!e^{\lambda}V_{m}(K)}{W(K)}-1.
\end{align}
\begin{proof}
The proof of \eqref{eqdTVforivgen} is completely analogous to the matroid counterpart (Corollary \ref{Corollary:Poissonmatroidsapprox}). 

\end{proof}

\end{Corollary}
The above result can be applied in the particular instances in which $K$ is a product  of convex sets, leading to the following quantitative approximations to the law of $Z_{K}$.
\begin{Corollary}\label{cor:rareeventsintrinsvol}
If $K_{1},K_{2},\dots$ is a sequence of convex sets in $\R^d$, for $d\geq 1$  and $K=K_{1}\times\cdots\times K_{n}$, then 
\begin{align}\label{eqrareeventsintrinsvol}
d_{TV}(\nu_{K},\gamma_K)
  &\leq  e^{\sum_{i=1}^{n}V_{1}(K_i)^2}-1.
\end{align}
Moreover, if $K_{i}=s_i\kappa_i$, with $s_i\in(0,1]$,  $\kappa_i\subset\R^{d}$ convex such that $\vartheta:=\sup_{i}W(\kappa_{i})<\infty$, then
\begin{align}\label{eqrareeventsintrinsvolv2}
d_{TV}(\nu_{K},\gamma_K)
  &\leq  e^{d\vartheta\sum_{i=1}^{n}s_i^2}-1.
\end{align}
If $K=[0,s_{1}]\times\dots\times [0,s_{n}]$, then
\begin{align}\label{eqrareeventsintrinsvol2}
d_{TV}(\nu_{K},\gamma_K)
  &\leq  e^{\sum_{i=1}^{n}s_{i}^2}-1.
\end{align}
\end{Corollary}
\begin{Remark}
If we choose the $s_{1},\dots, s_{n}$ in such a way that 
\begin{align*}
\lim_{n}\sum_{i=1}^nV_1(K_i)=\lambda 
\end{align*}
for some $\lambda>0$, at the same time that  $\min V_{1}(K_i)\rightarrow0$, then Corollary \ref{cor:rareeventsintrinsvol} implies that $\nu_{K}$ converges towards a Poisson distribution of parameter $\lambda$. That is to say, the intrinsic volumes random variables for rectangular intervals satisfy a law of rare events. 
\end{Remark}

\begin{proof}

By choosing $m=0$ in \eqref{eqdTVforivgen}, we deduce that
\begin{align*}
d_{TV}(\nu_{\varepsilon K},\gamma_{\varepsilon,K})
  &\leq e^{V_{1}(K)-\log(W(\varepsilon K))}-1.
\end{align*}

Define $G_{K}(\lambda):=W(\lambda K)$. By the direct product property, the mapping $\lambda\mapsto G_{K}(\lambda)$ satisfies $G_{K}(\lambda)=\prod_{i=1}^{n}G_{K_i}(\lambda)$. Taking logarithmic derivative and evaluating at zero, we deduce that
\begin{align*}
V_{1}(K)=V_{1}(K_1)+\cdots+V_{1}(K_n).
\end{align*}
In addition, by directly evaluating the identity $G_{K}(\lambda)=\prod_{i=1}^{n}G_{K_i}(\lambda)$ at $\lambda=1$, we get 
\begin{align}\label{eq:btrs1}
W(K)=W(K_1)\times \cdots\times W(K_n).
\end{align}
From here it follows that
\begin{align}\label{eq:btrs2}
d_{TV}(\nu_{\varepsilon K},\gamma_{\varepsilon,K})
  &\leq e^{\sum_{i=1}^nV_{1}(K_i)-\sum_{i=1}^{n}\log(1+V_{1}(K_i))}-1
	\leq e^{\sum_{i=1}^{n}V_{i}(K)^2}-1,
\end{align}
where in the second inequality we used \eqref{ineq:logTaylor}. Relation \eqref{eqrareeventsintrinsvol}  follows from here. Relation \eqref{eqrareeventsintrinsvol2} can be easily proved by observing that $V_{1}([0,s_{i}])=s_iV_{1}([0,1])=s_{i}$ (see \eqref{eq:Kcubedeffor}), which yields \eqref{eqrareeventsintrinsvol}.\\

\noindent To prove \eqref{eqrareeventsintrinsvolv2}, we use the homogeneity property of intrinsic volumes which establishes that $V_{\ell}(s_{i}\kappa_{i})=s_{i}^{\ell}V_{\ell}(\kappa_{i})$. We then use \eqref{eq:btrs1} and \eqref{eq:btrs2}, to deduce that
\begin{align*}
W(K_{i})
  &=1+s_iV_{1}(\kappa_{i})+s_{i}^2\sum_{\ell=2}^ds_{i}^{\ell-2}V_{\ell}(\kappa_i)\\
  &\leq1+s_iV_{1}(\kappa_{i})+d\theta s_{i}^2.
\end{align*}
Combining the above inequality with Corollary \ref{Cor:PoissonIVgenericresult}, we get 
\begin{align*}
d_{TV}(\nu_{\varepsilon K},\gamma_{\varepsilon,K})
  &\leq e^{\sum_{i=1}^nV_{1}(K_i)-\sum_{i=1}^{n}\log(W(K_i))}-1
	\leq e^{d\theta s_i^2}-1,
\end{align*}
as required.
\end{proof}

\begin{Remark}
Besides the simple form of the intrinsic volumes for the unit ball, the shape of our bound for the distance in total variation seems to not be applicable for this convex set. Furthermore, we have not found any evidence that might suggest that a Poisson approximation is appropriate in this circumstance.
\end{Remark}

\subsection{Log concavity with respect to the geometric distribution}\label{subsec:logcwrtcount}
In this subsection we present a quantification of the approximating error of geometric distributions and probability measures which are log-concave with respect to the counting measure over the positive integers (observe that this property is equivalent to log-concavity with respect to the geometric distribution). In the sequel, we will refer to such variables simply as 'log-concave'. 
As a consequence of Theorem \ref{Theoremmain:1}, we have the following geometric approximations in total variation 
\begin{Corollary}[Geometric approximation to sums of random variables]\label{cor:Geometricapprox}
Let $\{\xi_{i}\}_{i\in\N}$ be a collection of independent log-concave random variables. Suppose that the sequence $\alpha_{i}:=\Pb[\xi_{i}=0]$ consists of strictly positive terms and that the arithmetic mean of the terms $1/\alpha_i$, defined by
\begin{align}\label{eq:mundef}
\mathfrak{m}_n
  &:=\frac{1}{n}\sum_{i=1}^{n}\frac{1}{\alpha_i}
\end{align}
satisfies $\mathfrak{m}_n> 1+\frac{1}{n}$. Then, the sum $S_{n}=\sum_{i=1}^{n}\xi_i$ satisfies 
\begin{align*}
d_{TV}(S_{n},G_n)
  &\leq \frac{\mathfrak{m}_n-1}{\frac{1}{n}-(\mathfrak{m}_n-1)}
\end{align*}
where $G_n$ is a random variable with geometric distribution of parameter $1-n(\mathfrak{m}_n-1)$. In particular, if $n(\mathfrak{m}_n-1)\rightarrow0$, then 
\begin{align*}
d_{TV}(S_{n},G_n)
  &\leq Cn(\mathfrak{m}_n-1),
\end{align*}
for some constant $C>0$ independent of $n$.
\end{Corollary}

\begin{proof}
Using the fact that the log-concavity under convolution \cite{fekete1912problem}, we have that $S_{n}$ is log-concave. To verify the condition $\Pb[S_{n}=1]<\Pb[S_{n}=0]$, required for the application of Theorem \ref{Theoremmain:1}, we observe that  by \eqref{eq:Snprobtozeroone}, the product in the right hand side is equal to $n(\mathfrak{m}_n-1)$, so by the condition $\mathfrak{m}_n> 1+\frac{1}{n}$, we conclude that $\Pb[S_{n}=1]<\Pb[S_{n}=0]$, as required. From Theorem \ref{Theoremmain:1}, we thus conclude that

\begin{align*}
d_{K}(S_{n},G_n)
  &\leq   \frac{1}{1-n(\mathfrak{m}_n-1)}-1.
\end{align*}
The result easily follows from here.
\end{proof}

\subsection{Geometric approximation for infinitely divisible distributions}\label{subsecseomid} Let $ID(*)$ denote the set of probability measures consisting of the limits in distributions of $n$-fold convolutions of the form 
\begin{align*}
\mu_{n}*\cdots*\mu_{n},
\end{align*}
for given sequences of probability measures $\{\mu_{n}\}_{n\in\N}$ such that the above sequence of convolutions converge weakly. It is well known that the collection of measures $ID(*)$ are precisely those obtained as the limits in distribution of renormalized limits of compound Poisson random variables of the form 
$
\sum_{k=0}^N\xi_{k},$ where $N$ is a Poisson random variable and $\{\xi_{k}\}_{k}$ are independent and identically distributed distributions, convoluted with a Gaussian random variable. It is known from the basic theory of infinitely distributions (see \cite{MR1739520}), that every element of $\mu\in ID(*)$ is characterized by the property that 
\begin{align*}
\int_{\R}e^{\textbf{i}\lambda x}\mu(dx)
  &=\exp\{a-\frac{1}{2}\sigma^2\lambda^2-\int_{\R}(e^{\textbf{i}\lambda x}-1-\Indi{\{|x|\leq 1\}}\textbf{i}\lambda x)\Pi(dx)\},    
\end{align*}
where $a\in\R$, $\sigma>0$ and  $\Pi$ is a radon measure satisfying 
$\int_{\R}(1\wedge|x|)\Pi(dx)<\infty$. In the last decade, many efforts have been made to understand when does an infinitely divisible probability measure is log concave (see for instance \cite{MR3030616},  \cite{MR3126668} and \cite{MR2598065}). In the particular, in \cite[Theorem 4]{MR2598065}, the following result is proved
\begin{Theorem}\label{thm:mainidlogconc}
Let $X$ be a compound Poisson random variable of the form $X=\sum_{i=0}^{N}\xi_{i}$, where $N$ is a Poisson random variable with intensity $\lambda>0$ and the $\xi_i$ are independent and identically distributed random variables (independent of $N$), with log-concave probability mass function $F=\{F_{i}=\Pb[\xi_1=i]\}_{i\in\N}$ supported in $\N$. Then, $X$ is log-concave if and only if $\lambda F_{1}^2\geq 2F_{2}$.
\end{Theorem}
\begin{Remark}
If the distribution $F$ is fixed, for $\lambda$ large enough, the law of $X$ will be log-concave. In this sense, every log-concave random variable generates a family of infinitely divisible log-concave distributions.
\end{Remark}

Combining Theorem \ref{Theoremmain:1} and \ref{thm:mainidlogconc} allow us to prove the following corollary
\begin{corollary}[Geometric approximation for compound Poisson probability measures]\label{cor:geomapproxtoID}
Let $\nu$ be the probability measure associated to a random variable $X$ satisfying the conditions of  Theorem \ref{thm:mainidlogconc}. Then, provided that $\lambda F_1<1$, we have that 
\begin{align*}
d_{TV}(\nu,\mu)
  &\leq e^{\lambda(1-F_0)}-1
\end{align*}
where $\mu$ is a geometric random variable with parameter $1-\lambda F_1$.
\end{corollary}

\begin{proof}
In virtue of Theorem \ref{thm:mainidlogconc}, it suffices to verify that $\nu[1]/\nu[0]< 1$ and then to upper bound $\nu[0]/\mu[0]$, where $\mu$ is a geometric law with parameter $1-\nu[1]/\nu[0]$. Observe that 
\begin{align*}
\Pb[X=1]
  &=\sum_{n=1}^{\infty}nF_1F_0^{n-1}\frac{\lambda^n}{n!}e^{-\lambda}=\lambda F_1e^{-\lambda(1-F_0)}
\end{align*}
and 
\begin{align*}
\Pb[X=0]
  &=\sum_{n=0}^{\infty}F_0^{n}\frac{\lambda^n}{n!}e^{-\lambda}=e^{-\lambda(1-F_0)},
\end{align*}
so that $\Pb[X=1]/\Pb[X=0]=\lambda F_1$, which is less than one by hypothesis. From  Theorem \ref{Theoremmain:1}, it  follows that 
\begin{align*}
d_{TV}(\nu,\mu)
  &\leq e^{\lambda(1-F_0)}(1-\lambda F_1)-1.
\end{align*}
The result easily follows from here.\\
\end{proof}

\noindent \textit{The case of compound Geometric distributions}\\
In the spirit of Corollary \ref{cor:geomapproxtoID}, one could consider other type of distributions for the number of summands $N$. This direction of research has been considered by in Ninh and Pr\'{e}kopa,\cite{MR3126668}, where the following result was proved
\begin{Theorem}\label{thm:mainidlogconc2}
Suppose that $N$ has log-concave distribution over $\N$. Then, if $\{\xi_k\}_{k\in\N}$ are independent and identically distributed random variables with geometric distribution of parameter $p$, then the probability distribution of of $X=\sum_{k=1}^N\xi_k$ is log-concave.
\end{Theorem}

\noindent As in the case of Corollary \ref{cor:geomapproxtoID}, the combination of Theorems \ref{Theoremmain:1} and \ref{thm:mainidlogconc2}, yield the following result
\begin{corollary}\label{cor:Geometricpoissonlthm}
Let $X$ be as in Theorem \ref{thm:mainidlogconc2} and denote by $\nu$ its probability distribution. let $\{F_{i}\}_{i\geq0}$ be the probability point masses of $N$, namely, $F_{i}:=\Pb[N=i]$. Then, provided that $\Pb[X=1]/\Pb[X=0]<1$, we have the bound
\begin{align*}
d_{TV}(\nu,\mu)
  &\leq \frac{1}{F_1}\left(1+\frac{1-F_1}{p(1-p)}\right)^2-1.
\end{align*}
where $\mu$ is a geometric distribution of parameter $\Pb[X=1]/\Pb[X=0]$.
\end{corollary}
\begin{Remark}
If $p$ is fixed, the above result guarantees the existence of a constant $C>0$, depending on $p$, such that 
\begin{align*}
d_{TV}(\nu,\mu)
  &\leq C(1-F_1).
\end{align*}
\end{Remark}

\begin{proof}
Observe that 
\begin{align*}
\Pb[X=1]
  =\sum_{k=1}^{\infty}kF_kp(1-p)^{k},
\end{align*}
and
\begin{align*}
\Pb[X=0]
  =\sum_{k=0}^{\infty}F_k(1-p)^{k}.
\end{align*}
The condition $\nu[1]/\nu[0]<1$ follows from the hypotheses of the corollary. Observe that by the condition $F_k\leq 1-F_1$, which is valid for all $k\neq 1$, we have the inequality 
\begin{align*}
\Pb[X=0]
  \leq F_1(1-p)+(1-F_1)/p.
\end{align*}
By similar arguments, 
\begin{align*}
\Pb[X=1]
  &\leq F_1p(1-p)+(1-F_1)p\sum_{k=1}^{\infty}k(1-p)^k\\
  &=  F_1p(1-p)+(1-F_1)(1-p)/p.
\end{align*}
In particular, 
\begin{align*}
1-\frac{\Pb[X=1]}{\Pb[X=0]}
  \geq \Pb[X=0]^{-1}(F_{1}(1-p)^2),
\end{align*}
which implies that 
\begin{align*}
\nu[0]/\mu[0]
  &\leq (F_1(1-p)+(1-F_1)/p)^2(F_{1}(1-p)^2)^{-1}\\
  &\leq ((1-p)+(1-F_1)/p)^2(F_{1}(1-p)^2)^{-1}.
\end{align*}
The result easily follows from here.
\end{proof}

\noindent \textit{Regarding Poissonian log-concavity for infinitely divisible distributions}\\
In view of Section \ref{subsecseomid}, it is natural to wonder weather or not one could apply techniques of log concavity to the analysis of Poisson approximations for infinitely divisible probability measures. This would then lead us to ask ourselves if we could prove log-concavity with respect to the Poisson distributions for infinitely divisible distributions. An answer to this question was addressed by Yu in \cite[Proposition 3]{MR2598065}, where it was proved (among other things), that every compound Poisson probability measure that is log-concave with respect to a Poisson law is the Poisson distribution. This would then imply that the hypothesis we require for the application of the  techniques presented in this paper for proving Poisson approximation, will not hold except from the trivial case where the underlying probability measure is already a Poisson distribution.

\section{Proof of Theorem \ref{Theoremmain:1}}\label{Proofmain}
\begin{proof}
Let $A\in\mathcal{B}(\N_{0})$ and define $\mathcal{K}:=\{k\in\Z\ ;\ q_{k}\geq p_{k}\}$. Then, 
\begin{align*}
d_{TV}(X,Y)
  &=\sum_{k\in \mathcal{K}}(q_{k}-p_{k}).
\end{align*}
Next we find a suitable bound for $q_{k}-p_{k}$, by splitting into the cases $k\geq\ell$ and $k<\ell$.\\

\noindent \textbf{Step I}\\
First we handle the case $k\geq \ell$. 
Let $\psi:\R\rightarrow\R\backslash\{0\}$ denote the function 
\begin{align}\label{eq:psiauxdef}
\psi(x):=x\Indi{\{x\neq 0\}}+\Indi{\{x= 0\}}.
\end{align}
Consider the sequence $\{a_{k}\}_{k\geq 1}$ given by 
\begin{align*}
a_{k}
  &:=	\frac{\psi(p_{k-1})\psi(p_{k+1})}{\psi(p_{k})^2},
\end{align*}
and define $\{f_{k}\}_{k\geq 1}$ and $\{g_{k}\}_{k\geq 1}$ as  
\begin{align*}
f_{k}
  :=\frac{\psi(p_{k})}{\psi(p_{k-1})}
\ \ \ \ \ \ \text{ and}\ \ \ \ \ \ 
g_{k}
  :=
	\frac{\psi(q_{k})}{\psi(q_{k-1})}.
\end{align*}
Using induction, we will show that 
\begin{align}\label{eq:midgoal}
g_{k}
  \leq g_{\ell+1}\prod_{j=\ell+1}^{k-1} a_{j} 
	\ \ \ \ \ \ \text{ and}\ \ \ \ \ \ 
f_{k}
  = f_{\ell+1}\prod_{j=\ell+1}^{k-1} a_{j}	,
\end{align}
where we define the product above as one in the case where $k\leq\ell+1$. To prove this, let's consider first the base cases $k=\ell$, $k=\ell+1$ and $k=\ell+2$. When $k\in\{\ell,\ell+1\}$, the relations hold by definition. For the case $k=\ell+2$, the above relations are equivalent to 
\begin{align*}
g_{\ell+2}
  \leq g_{\ell+1} a_{\ell+1}
	\ \ \ \ \ \ \text{ and}\ \ \ \ \ \ 
f_{\ell+2}
  = f_{\ell+1}a_{\ell+1}.
\end{align*}
Analyzing individually the eight different possibilities for the vector 
$$(\Indi{\{p_{\ell}>0\}},\Indi{\{p_{\ell+1}>0\}},\Indi{\{p_{\ell+2}>0\}}),$$
we can compute explicitly the values of $f_{\ell+2}$, $p_{\ell},p_{\ell+1}$ and $p_{\ell+2}$, leading to the identity $f_{\ell+1}=f_{\ell+1}a_{\ell+1}$. For proving 
$g_{\ell+2}\leq g_{\ell+1} a_{\ell+1}$, we define  
\begin{align*}
\tilde{a}_{l}
  &:=	\frac{\psi(q_{l-1})\psi(q_{l+1})}{\psi(q_{l})^2}.
\end{align*}
We will show that $\tilde{a}_{l}\leq a_{l}$ for all $l\in\Z$. Since $a_l,\tilde{a}_l>0$ for all $l\in\Z$, this statement is equivalent to proving that 
\begin{align*}
\log(\psi(q_{l-1}))+\log(\psi(q_{l+1}))-2\log(\psi(q_{l}))
  &\leq \log(\psi(p_{l-1}))+\log(\psi(p_{l+1}))-2\log(\psi(p_{l})).
\end{align*}
We can easily check that the above relation is equivalent to 
\begin{align}\label{eq:eqone}
\frac{1}{2}(\log(\psi(q_{k-1})/\psi(p_{k-1}))+\log(\psi(q_{k+1})/\psi(p_{k+1})))
  &\leq \log(\psi(q_{k})/\psi(p_{k})).
\end{align}
Observe that for every $j\in\Z$, 
\begin{align*}
\log(\psi(q_{j})/\psi(p_{j}))
  &=\Phi_{\mu,\nu}(j).
\end{align*}
To verify this we subdivide into the cases $q_{j}>0$ and $q_{j}=0$. Thus, relation \eqref{eq:eqone} is equivalent to 
\begin{align*}
\frac{1}{2}(\Phi_{\mu,\nu}(k-1)+\Phi_{\mu,\nu}(k+1))
  &\leq \Phi_{\mu,\nu}(k),
\end{align*}
which holds due to \eqref{eq:convexityofPhi}. The proof of the inequality $\tilde{a}_{k}\leq a_k$ is now complete. From here we deduce that
\begin{align*}
g_{\ell+1}a_{\ell+1}
  &\geq g_{\ell+1}\tilde{a}_{\ell+1}.
\end{align*}
Analyzing individually the eight different possibilities for the vector 
$$(\Indi{\{q_{\ell}>0\}},\Indi{\{q_{\ell+1}>0\}},\Indi{\{q_{\ell+2}>0\}}),$$
we can check that $g_{\ell+1}\tilde{a}_{\ell+1}=g_{\ell+2}$ and consequently, 
\begin{align*}
g_{\ell+2}
  &\leq g_{\ell+1}a_{\ell+1}.
\end{align*}
Now we assume that \eqref{eq:midgoal} holds for a given $k\geq \ell+2$ and show the result for $k+1$. Using the induction hypothesis with $\ell$ replaced by $k-2$, we get 
\begin{align*}
g_{k}
  \leq g_{k-1} a_{k-1} 
	\ \ \ \ \ \ \text{ and}\ \ \ \ \ \ 
f_{k}
  = f_{k-1}a_{k-1}.
\end{align*}
A further application of the induction to the terms $g_{k-1}$ and $f_{k-1}$ leads to the required result.\\

\noindent We have thus proved \eqref{eq:midgoal}. Now, we observe that \eqref{eq:midgoal} implies that for all $k\geq \ell+1$,
\begin{align*}
\psi(p_{k})
  &=\psi(p_{k-1})f_{k}
	=\psi(p_{k-1})f_{\ell+1}\prod_{j=\ell+1}^{k-1} a_{j}
	\geq \frac{f_{\ell+1}}{g_{\ell+1}}\psi(p_{k-1})g_k.
\end{align*} 
Assume that the support of $\mu$ is of the form $[a,b]$, with $a\in\Z\cup\{-\infty\}$ and $b\in\Z\cup\{\infty\}$. Applying inductively the previous inequality we can show that 
\begin{align*}
\psi(p_{k})
  &\geq \left(\frac{f_{a\vee \ell+1}}{g_{a\vee \ell+1}}\right)^{k-a\vee \ell}\psi(p_{a\vee \ell})\prod_{j=\ell+1}^{k}g_j
	= \left(\frac{f_{a\vee \ell+1}}{g_{a\vee \ell+1}}\right)^{k-a\vee \ell}\frac{\psi(p_{a\vee \ell})}{\psi(q_{a\vee \ell})}\psi(q_{k}).
\end{align*} 
Due to the condition $q_{\ell}>0$ and the absolute continuity of $\nu$ with respect to $\mu$, it holds that $\ell\geq a$, so we can conclude that for every $k\geq \ell$,
\begin{align*}
0\leq q_{k}-p_{k}
  &\leq \bigg{(}1-\frac{p_{\ell}}{q_{\ell}}\left(\frac{f_{\ell+1}}{g_{\ell+1}}\right)^{k-\ell}\bigg{)} q_{k}.
\end{align*} 
and 
\begin{align*}
0\leq q_{k}-p_{k}
  &\leq \bigg{(}\frac{q_{\ell}}{p_{\ell}}\left(\frac{g_{\ell+1}}{f_{\ell+1}}\right)^{k-\ell}-1\bigg{)} p_{k}.
\end{align*} 
From here we conclude that
\begin{align*}
0\leq q_{k}-p_{k}
  &\leq \bigg{(}1-\frac{p_{\ell}}{q_{\ell}}\left(\frac{p_{\ell+1}q_{\ell}}{p_{\ell}q_{\ell+1}}\right)^{k-\ell}\bigg{)} q_{k}.
\end{align*} 
and 
\begin{align*}
0\leq q_{k}-p_{k}
  &\leq \bigg{(}\frac{q_{\ell}}{p_{\ell}}\left(\frac{p_{\ell}q_{\ell+1}}{p_{\ell+1}q_{\ell}}\right)^{k-\ell}-1\bigg{)} p_{k}.
\end{align*} 

\noindent \textbf{Step II}\\
To handle the case $k<\ell$, we consider the function $\vartheta:\Z\rightarrow\Z$ given by the reflection over the value $\ell$, namely,
\begin{align*}
\vartheta(k)
  &:=\left\{
	\begin{array}{lll}
	\ell-|k-\ell| & \text{ if } & k\geq\ell\\
	\ell+|k-\ell| & \text{ if } & k<\ell.
	\end{array}
	\right.
\end{align*}
Define $\tilde{\mu}$ and $\tilde{\nu}$ as the pullback measures of $\mu$ and $\nu$ with respect to $\vartheta$. Let $\tilde{p}_k$ and $\tilde{q}_k$ be the probability masses at $k$ of the measures $\tilde{\mu}$ and $\tilde{\nu}$. It is not hard to show that $\tilde{\mu}$ and $\tilde{\nu}$ satisfy the hypothesis of Theorem \ref{Theoremmain:1}. Therefore, by replacing $\ell$ by $\ell-1$ in our conclusion, we deduce that for every $j$ satisfying $\tilde{q}_j>0$,
\begin{align*}
0\leq \tilde{q}_{j}-\tilde{p}_{j}
  &\leq \bigg{(}1-\frac{\tilde{p}_{\ell-1}}{\tilde{q}_{\ell-1}}\left(\frac{\tilde{f}_{\ell}}{\tilde{g}_{\ell}}\right)^{j-\ell+1}\bigg{)} \tilde{q}_{j}.
\end{align*} 
In particular, by taking $j=\vartheta(k)$, for $k<\ell$, we conclude that  for all $k\in\mathcal{K}$,
\begin{align*}
0\leq q_{k}-p_{k}
  &\leq \bigg{(}1-\frac{p_{\ell+1}}{q_{\ell+1}}\left(\frac{p_{\ell}q_{\ell+1}}{p_{\ell+1}q_{\ell}}\right)^{j-\ell+1}\bigg{)} q_{k}
	= \bigg{(}1-\frac{p_{\ell}}{q_{\ell}}\left(\frac{p_{\ell+1}q_{\ell}}{p_{\ell}q_{\ell+1}}\right)^{k-\ell}\bigg{)} q_{k}.
\end{align*} 
Proceeding analogously, we can show that 
\begin{align*}
0\leq q_{k}-p_{k}
  &\leq \bigg{(}\frac{q_{\ell}}{p_{\ell}}\left(\frac{p_{\ell+1}q_{\ell}}{p_{\ell}q_{\ell+1}}\right)^{k-\ell}-1\bigg{)} p_{k}.
\end{align*} 

The result easily follows by summing over $k\in\mathcal{K}$ in the inequalities 
\begin{align*}
0\leq q_{k}-p_{k}
  &\leq \bigg{(}1-\frac{p_{\ell}}{q_{\ell}}\left(\frac{p_{\ell+1}q_{\ell}}{p_{\ell}q_{\ell+1}}\right)^{|k-\ell|}\bigg{)} q_{k}.
\end{align*} 
and 
\begin{align*}
0\leq q_{k}-p_{k}
  &\leq \bigg{(}\frac{q_{\ell}}{p_{\ell}}\left(\frac{p_{\ell}q_{\ell+1}}{p_{\ell+1}q_{\ell}}\right)^{|k-\ell|}-1\bigg{)} p_{k}.
\end{align*}

\end{proof}

\section{The continuous regime}\label{sec:continuouspart}
Next we present the continuous counterparts to the results presented in the previous section. Our main results are Theorem \ref{Theoremmain:02} and Theorem \ref{Theoremmain:2}

We say that $\mu$ is log-concave with respect to $\mu$ if $\nu$ is absolutely continuous with respect to $\mu$ and the Radon-Nikodym derivative $\frac{d\nu}{d\mu}$ is such that the mapping $x\mapsto\Phi_{\mu,\nu}(x)$, with
\begin{align*}
\Phi_{\mu,\nu}(x)
  &:=\left\{\begin{array}{ccc}
\log( \frac{d\nu}{d\mu}(x))  & \text{ if } & f_{\nu}(x)\neq0\\
-\infty  & \text{ if } & f_{\nu}(x)=0
\end{array}\right.
\end{align*}
is concave in the sense that for all $x,y\in\R$ and $\lambda\in[0,1]$, 
\begin{align*}
\Phi_{\mu,\nu}((1-\lambda)x+\lambda y)
  &\geq (1-\lambda)\Phi_{\mu,\nu}(x)+\lambda\Phi_{\mu,\nu}(y).
\end{align*}

\begin{Theorem}\label{Theoremmain:02}
Suppose that $\mu$ is the Lebesgue measure in $R_{+}:=\R\backslash\{0\}$ and $\nu$ is a probability measure, log-concave with respect to $\mu$. Define $f_{\nu}(x):=\frac{d\nu}{d\mu}(x)$. Suppose that $f_{\nu}$ is differentiable in $x$ and satisfies $f_{\nu}^{\prime}(0)<0$, and $f_{\nu}(0)\neq0$. Then, if $\gamma$ is an exponential distribution with 
\begin{align*}
\gamma[(x,\infty)]
  &=e^{\frac{f_{\nu}^{\prime}(0)}{f_{\nu}(0)} x},
\end{align*}
and $f_{\gamma}(x):=\frac{d\gamma}{d\mu}(x)$, it holds that
\begin{align*}
d_{K}(\nu,\gamma)
  &\leq \frac{f_{\nu}(0)}{f_{\gamma}(0)}-1.
\end{align*}
\end{Theorem}
\begin{proof}
Let $X$ and $Y$ be random variables with probability distributions $\nu$ and $\gamma$, respectively. Observe that by the log-concavity condition of $\nu$ with respect to the Lebesgue measure, we have that for all $x\in\R_{+}$, 
\begin{align*}
\frac{f_{\nu}^{\prime}(x)}{f_{\nu}(x)}
  &=\frac{d}{dx}\log(f_{\nu}(x))
  \leq \frac{d}{dx}\log(f_{\nu}(y))|_{y=0}=\frac{f_{\nu}^{\prime}(0)}{f_{\nu}(0)}=-f_{\gamma}(0).
\end{align*}
Thus, by Gr\"onwall's inequality,
\begin{align*}
f_{\nu}(x)
  &\leq f_{\nu}(0)e^{-f_{\gamma}(0)x}=\frac{f_{\nu}(0)}{f_{\gamma}(0)}f_{\gamma}(x).
\end{align*}
From here we conclude that if $X$ and $Y$ are $\R_{+}$-valued random variables with probability distributions $\mu$ and $\gamma$ respectively, then 
\begin{align*}
\Pb[X\leq z]-\Pb[Y\leq z]
  &\leq \int_0^z(\frac{f_{\nu}(0)}{f_{\gamma}(0)}-1)f_{\gamma}(x)dx\leq \frac{f_{\nu}(0)}{f_{\gamma}(0)}-1,
\end{align*}
and 
\begin{align*}
\Pb[Y\leq z]-\Pb[X\leq z]
  &=\Pb[X> z]-\Pb[Y> z]\leq \int_z^{\infty}(\frac{f_{\nu}(0)}{f_{\gamma}(0)}-1)f_{\gamma}(x)dx\leq \frac{f_{\nu}(0)}{f_{\gamma}(0)}-1.
\end{align*}
The result follows from here.
\end{proof}

\begin{Theorem}\label{Theoremmain:2}
Suppose that $\nu$ and $\mu$ are probability measures absolutely continuous with respect to the Lebesgue measure and $\nu$ is log-concave with respect to $\mu$ and the support of $\mu$ is connected. Then, 
\begin{align*}
d_{TV}(\mu,\nu)
  &\leq \int_{\R}((f_{\nu}(z)/f_{\mu}(z))e^{(x-z)(f_{\nu}^{\prime}(z)/f_{\nu}(z)-f_{\mu}^{\prime}(z)/f_{\mu}(z))}-1)_{+}\mu(dx).
\end{align*}
and 
\begin{align*}
d_{TV}(\mu,\nu)
  &\leq \int_{\R}(1-(f_{\mu}(z)/f_{\nu}(z))e^{-(x-z)(f_{\nu}^{\prime}(z)/f_{\nu}(z)-f_{\mu}^{\prime}(z)/f_{\mu}(z))})_{+}\nu(dx).
\end{align*}
In particular, if $f_{\nu}^{\prime}(z)/f_{\nu}(z)=f_{\mu}^{\prime}(z)/f_{\mu}(z)$, the above inequalities imply that $f_{\nu}(z)/f_{\mu}(z)\geq 1$ and 
\begin{align*}
d_{TV}(\mu,\nu)
  &\leq (f_{\nu}(z)/f_{\mu}(z)-1)\wedge (1-f_{\mu}(z)/f_{\nu}(z)).
\end{align*}
\end{Theorem}

\begin{proof}
Define the functions $f_{\nu}(x):=\frac{d\nu}{d \mu}(x)$ and $f_{\mu}(x):=\frac{d\mu}{d \mu}(x)$, as well as the set $\mathcal{X}:=\{x\in\R\ ;\ f_{\nu}(x)>f_{\mu}(x)\}$. Notice that 
\begin{align*}
d_{TV}(\nu,\mu)
  &=\int_{\mathcal{X}}(f_{\nu}(x)-f_{\mu}(x))dx,
\end{align*}
so it suffices to find an upper bound for $f_{\nu}(x)-f_{\mu}(x)$. Let $\psi$ be given by \eqref{eq:psiauxdef}. Observe that for all $x\in\mathcal{X}$,
\begin{align*}
\Phi_{\mu,\nu}(x)
  &=\log\circ\psi(f_{\nu}(x))-\log\circ\psi(f_{\mu}(x)).
\end{align*}
By the concavity of $\Phi_{\mu,\nu}(x)$ and the differentiability condition over $f_{\nu}(x)$ and $f_{\mu}(x)$, we have that for all $x\in\mathcal{X}$, 
\begin{align*}
\Phi_{\mu,\nu}(x)-\Phi_{\mu,\nu}(z)
  &=\int_z^x\Phi_{\mu,\nu}^{\prime}(y)dy
	= \int_z^x\int_z^{y}\Phi_{\mu,\nu}^{\prime\prime}(w)dwdy+(x-z)\Phi_{\mu,\nu}^{\prime}(z) \leq  (x-z)\Phi_{\mu,\nu}^{\prime}(z).
\end{align*}
Since $f_{\nu}(z)f_{\nu}(x)>0$, then 
\begin{align*}
\log(f_{\nu}(x)/f_{\mu}(x))
  &\leq \log(f_{\nu}(z)/f_{\mu}(z))+ (x-z)(f_{\nu}^{\prime}(z)/f_{\nu}(z)-f_{\mu}^{\prime}(z)/f_{\mu}(z)).
\end{align*}
By taking exponential, we thus get that 
\begin{align*}
f_{\nu}(x)
  &\leq f_{\mu}(x)(f_{\nu}(z)/f_{\mu}(z))e^{(x-z)(f_{\nu}^{\prime}(z)/f_{\nu}(z)-f_{\mu}^{\prime}(z)/f_{\mu}(z))}.
\end{align*}
Consequently, we get 
\begin{align*}
f_{\nu}(x)-f_{\mu}(x)
  &\leq ((f_{\nu}(z)/f_{\mu}(z))e^{(x-z)(f_{\nu}^{\prime}(z)/f_{\nu}(z)-f_{\mu}^{\prime}(z)/f_{\mu}(z))}-1)f_{\mu}(x),
\end{align*}
and 
\begin{align*}
f_{\nu}(x)-f_{\mu}(x)
  &\leq (1-(f_{\mu}(z)/f_{\nu}(z))e^{-(x-z)(f_{\nu}^{\prime}(z)/f_{\nu}(z)-f_{\mu}^{\prime}(z)/f_{\mu}(z))})f_{\nu}(x).
\end{align*}
The result easily follows from here by taking positive part in both sides and then integrating over $\R$. 
\end{proof}

%
The following Corollary follows from a direct application of Theorem \ref{Theoremmain:2}.
\begin{Corollary}
Suppose that $\nu$ is a Gamma distribution shape parameter $\kappa_1>0$ and scale parameter $\lambda_1$ while $\gamma$ is a Gamma distribution with shape parameter $\lambda_2$ and shape parameter $\kappa_2$. 
\begin{enumerate}
\item[(i)] If $\frac{\kappa_1-\kappa_2}{\lambda_1-\lambda_2}>0$, 
\begin{align*}
d_{TV}(\nu,\gamma)
  & \leq \left(\frac{\lambda_1^{\kappa_1}}{\lambda_2^{\kappa_2}}\left(\frac{\kappa_1-\kappa_2}{\lambda_1-\lambda_2}\right)^{\kappa_1-\kappa_2}e^{-(\kappa_1-\kappa_2)}-1\right)
	\wedge \left(1-\frac{\lambda_2^{\kappa_2}}{\lambda_1^{\kappa_1}}\left(\frac{\lambda_1-\lambda_2}{\kappa_1-\kappa_2}\right)^{\kappa_1-\kappa_2}e^{-(\kappa_2-\kappa_1)}\right)
\end{align*}
\item[(ii)] If we instead assume the existence of a constant $z>0$ such that  
\begin{align*}
\frac{\kappa_1-\kappa_2}{z}+\lambda_2-\lambda_1\leq \frac{\lambda_2}{4},
\end{align*}
we obtain the more general bound
\begin{align*}
d_{TV}(\mu,\nu)
  &\leq |(z/e)^{\kappa_1-\kappa_2}-1|\\
	&+ (z/e)^{\kappa_1-\kappa_2}(1+\kappa_1+\kappa_2)2^{\kappa_1+1}
	\left|\frac{\kappa_1-\kappa_2}{\lambda_2z}+\frac{\lambda_2-\lambda_1}{\lambda_2} \right|^{\frac{1}{2}}.
\end{align*}
In particular, if $\kappa_1=\kappa_2=\kappa$ for some $\kappa>0$ and $4(\lambda_2-\lambda_1)\leq \lambda_2$, then 
\begin{align*}
d_{TV}(\mu,\nu)
  &\leq (1+2\kappa)2^{\kappa+1}
	\left|\frac{\lambda_2-\lambda_1}{\lambda_2} \right|^{\frac{1}{2}},
\end{align*}
and if $\lambda_1=\lambda_2=\kappa$ for some $\lambda>0$, then by choosing $z=4/\lambda_2$ in the identity above, we get 
\begin{align*}
d_{TV}(\mu,\nu)
  &\leq |(e\lambda/4)^{\kappa_2-\kappa_1}-1|\\
	&+ (e\lambda/4)^{\kappa_2-\kappa_1}(1+\kappa_1+\kappa_2)2^{\kappa_1+1}
	|\kappa_1-\kappa_2|^{\frac{1}{2}}.
\end{align*}
Observe that if the underlying parameters are contained in a compact interval $K$ of $\R_{+}\backslash\{0\}$, we can guarantee the existence of a constant $C$ only depending on $K$, such that  
\begin{align*}
d_{TV}(\mu,\nu)
  &\leq C\sqrt{|\lambda_1-\lambda_2|},
\end{align*}
when $\kappa_1=\kappa_2$ and 
\begin{align*}
d_{TV}(\mu,\nu)
  &\leq C\sqrt{|\kappa_1-\kappa_2|}.
\end{align*}
when $\lambda_1=\lambda_2$

\end{enumerate}

\end{Corollary}

\begin{proof}
Define 
\begin{align}\label{eq:zdefexample}
z
  &:=\frac{\kappa_1-\kappa_2}{\lambda_1-\lambda_2},
\end{align}
which is strictly positive by hypothesis. One can easily check that $\nu$ is log concave with respect to $\mu$ and
\begin{align}\label{eq:fnuoverfmu}
f_{\nu}^{\prime}(z)/f_{\nu}(z)=\frac{\kappa_1-1}{z}-\lambda_1
=\frac{\kappa_2-1}{z}-\lambda_2=f_{\gamma}^{\prime}(z)/f_{\gamma}(z),
\end{align}
where the second identity follows from the fact that $z$ is given by \eqref{eq:zdefexample}. Using the fact that
\begin{align*}
f_{\nu}(z)/f_{\gamma}(z)
  &=z^{\kappa_1-\kappa_2}e^{-z(\lambda_1-\lambda_2)}
	=\frac{\lambda_1^{\kappa_1}}{\lambda_2^{\kappa_2}}\left(\frac{\kappa_1-\kappa_2}{\lambda_1-\lambda_2}\right)^{\kappa_1-\kappa_2}e^{-(\kappa_1-\kappa_2)}
\end{align*}
as well as Theorem \ref{Theoremmain:2}, we thus conclude that 

\begin{align*}
d_{TV}(\nu,\gamma)
  &\leq \left(\frac{\lambda_1^{\kappa_1}}{\lambda_2^{\kappa_2}}\left(\frac{\kappa_1-\kappa_2}{\lambda_1-\lambda_2}\right)^{\kappa_1-\kappa_2}e^{-(\kappa_1-\kappa_2)}-1\right)
	\wedge \left(1-\frac{\lambda_2^{\kappa_2}}{\lambda_1^{\kappa_1}}\left(\frac{\lambda_1-\lambda_2}{\kappa_1-\kappa_2}\right)^{\kappa_1-\kappa_2}e^{-(\kappa_2-\kappa_1)}\right),
\end{align*}
as required.\\

\noindent For handling the case $\kappa_1=\kappa_2$, we use \eqref{eq:fnuoverfmu} to deduce that 
\begin{align*}
f_{\nu}^{\prime}(z)/f_{\nu}(z)-f_{\gamma}^{\prime}(z)/f_{\gamma}(z)
  &=\frac{\kappa_1-\kappa_2}{z}+\lambda_2-\lambda_1
\end{align*}
for every $z\geq0$. In addition, we have that 
\begin{align*}
f_{\nu}(z)/f_{\gamma}(z)
  &= z^{\kappa_1-\kappa_2}e^{-z(\lambda_1-\lambda_2)},
\end{align*}
so we deduce that for all $z>0$,
\begin{align*}
(f_{\nu}(z)/f_{\gamma}(z))e^{(x-z)(f_{\nu}^{\prime}(z)/f_{\nu}(z)-f_{\gamma}^{\prime}(z)/f_{\gamma}(z))}
  &=(z/e)^{\kappa_1-\kappa_2}e^{x(\frac{\kappa_1-\kappa_2}{z}+\lambda_2-\lambda_1)},
\end{align*}
so by Theorem \ref{Theoremmain:2},

\begin{align*}
d_{TV}(\mu,\nu)
  &\leq \int_{\R}((z/e)^{\kappa_1-\kappa_2}e^{x(\frac{\kappa_1-\kappa_2}{z}+\lambda_2-\lambda_1)}-1)_{+}\gamma(dx)
	\leq \E[|(z/e)^{\kappa_1-\kappa_2}e^{X(\frac{\kappa_1-\kappa_2}{z}+\lambda_2-\lambda_1)}-1|],
\end{align*}
where $X$ is Gamma distributed with parameters $\lambda_2$ and $\kappa_2$, as required. 
Adding and substracting the term $(z/e)^{\kappa_1-\kappa_2}$ in the right hand side, and then applying the triangle inequality and the Cauchy-Scharz inequality, we thus obtain 
\begin{align*}
d_{TV}(\mu,\nu)
  &\leq |(z/e)^{\kappa_1-\kappa_2}-1|+(z/e)^{\kappa_1-\kappa_2}\E[|e^{X(\frac{\kappa_1-\kappa_2}{z}+\lambda_2-\lambda_1)}-1|^2]^{\frac{1}{2}},
\end{align*}
which after suitable algebraic manipulations, yields 
\begin{align*}
d_{TV}(\mu,\nu)
  &\leq |(z/e)^{\kappa_1-\kappa_2}-1|\\
	&+ (z/e)^{\kappa_1-\kappa_2}|\left( 1-\frac{1}{\lambda_2}(\frac{\kappa_1-\kappa_2}{z}+\lambda_2-\lambda_1)\right)^{-\kappa_1}-1|\\
	&+(z/e)^{\kappa_1-\kappa_2}\left(\left( 1-\frac{2}{\lambda_2}(\frac{\kappa_1-\kappa_2}{z}+\lambda_2-\lambda_1)\right)^{-\kappa_1}
	-\left( 1-\frac{1}{\lambda_2}(\frac{\kappa_1-\kappa_2}{z}+\lambda_2-\lambda_1)\right)^{-2\kappa_1}
	\right)^{1/2},
\end{align*}
due to the fact that $z$ satisfies 

\begin{align*}
\frac{1}{\lambda_2}(\frac{\kappa_1-\kappa_2}{z}+\lambda_2-\lambda_1)\leq \frac{1}{4}
\end{align*}
An application of the mean value theorem thus yields 
\begin{align*}
d_{TV}(\mu,\nu)
  &\leq |(z/e)^{\kappa_1-\kappa_2}-1|\\
	&+ (z/e)^{\kappa_1-\kappa_2}\kappa_12^{\kappa_1+1}|\frac{1}{\lambda_2}(\frac{\kappa_1-\kappa_2}{z}+\lambda_2-\lambda_1)|\\
	&+(z/e)^{\kappa_1-\kappa_2}\left(2^{\kappa_1+1}(\kappa_1+\kappa_2)|\frac{1}{\lambda_2}(\frac{\kappa_1-\kappa_2}{z}+\lambda_2-\lambda_1)|
	\right)^{1/2}.
\end{align*}
The result follows from here
\begin{align*}
d_{TV}(\mu,\nu)
  &\leq |(z/e)^{\kappa_1-\kappa_2}-1|\\
	&+ (z/e)^{\kappa_1-\kappa_2}(1+\kappa_1+\kappa_2)2^{\kappa_1+1}
	|\frac{\kappa_1-\kappa_2}{\lambda_2z}+\frac{\lambda_2-\lambda_1}{\lambda_2}|^{\frac{1}{2}}.
\end{align*}

\end{proof}

\nocite{AlbHu}

\bibliographystyle{plain}
\bibliography{limittheorems}

\end{document}